\documentclass[11pt]{amsart}

\usepackage[utf8]{inputenc}
\usepackage[T1]{fontenc}
\usepackage[francais,english]{babel}
\usepackage{amsmath,amsfonts,amssymb,amsthm}
\usepackage{mathrsfs}
\usepackage{graphicx}
\usepackage{graphics}

\newcommand{\E}{{\mathbb{E}}}
\newcommand{\Z}{{\mathbb{Z}}}

\newcommand{\Prob}{{\mathbb{P}}}
\newcommand{\1}{{\mathbf{1}}}
\let\mathcal\mathscr
\let\geq\geqslant
\let\leq\leqslant

\renewcommand\ge{\geq}

\def\AA{\mathcal{A}}
\def\ee{\mathrm{e}}

\def\var{\mathrm{var}}
\def\cov{\mathrm{cov}}

\theoremstyle{plain}
\newtheorem{theo}{Theorem}[section]
\newtheorem{coro}[theo]{Corollary}
\newtheorem{defi}[theo]{Definition}
\newtheorem{prop}[theo]{Proposition}
\newtheorem{lemm}[theo]{Lemma}
\newtheorem{conj}[theo]{Conjecture}
\theoremstyle{remark}

\makeatletter
\@addtoreset{equation}{section}

\begin{document}

\title[Phylogenetic distances for models with influence]{Phylogenetic distances for neighbour dependent substitution processes}
\author{Mikael Falconnet}
\address{Université Joseph Fourier Grenoble 1\\
Institut Fourier UMR 5582 UJF-CNRS\\
100 rue des Maths, BP 74\\
38402 Saint Martin d'Hères\\
France }
\date{\today}



\begin{abstract}
  We consider models of nucleotidic substitution processes where the
  rate of substitution at a given site depends on the state of the
  neighbours of the site. 
  We first
  estimate the time elapsed between an ancestral sequence at
  stationarity and a present sequence. Second,
  assuming that two sequences are issued from a common
  ancestral sequence at stationarity, we estimate the time since
   divergence. In the simplest nontrivial case of a
  Jukes-Cantor model with CpG influence, we provide and justify mathematically
  consistent estimators in these two
  settings. We also provide asymptotic confidence intervals, valid for
  nucleotidic sequences of finite length, and we compute explicit
  formulas for the estimators and for their confidence intervals. In
  the general case of an RN model with YpR influence, we extend these
  results under a proviso, namely that the equation defining
  the estimator has a unique solution.
\end{abstract}



\maketitle

\section*{Introduction}

A crucial step in the computation of phylogenetic trees based on
aligned DNA sequences is the estimation of the evolutionary times
between these sequences. In most phylogenetic algorithms based on stochastic substitution models, one assumes
that each site evolves independently from the others and, in general,
according to a given Markovian kernel. This assumption is mainly due to the difficulty to work without the assumption of independence. To understand why, note that the distribution of the nucleotide at site $i$ at a given time depends a priori on the values at previous times of the dinucleotides at sites $i-1$ and $i+1$, whose joint distributions, in turn, may depend on the values of some trinucleotides, and so on. Hence, one is faced with infinite-dimensional linear systems, which are generically hard to solve. 
Besides, the magnitude of the effect  of the neighbours on the substitution rates can be large.
Since some neighbour influences are well documented in the literature,
and caused by well known biological mechanisms, it seems necessary to take into account the neighbour influences in substitution models. To wit, a class
of mathematical models with neighbour influences was recently introduced by
biologists, see \cite{galt:phylo}, and studied mathematically, see \cite{piau:solv}. 

The goal of the present paper is to show that one can compute consistent estimators of
the distances between DNA sequences whose evolution is ruled by models with influence in a specific class of models.

We completely describe the construction in the simplest
non trivial case, the Jukes-Cantor model with (symmetric) CpG
influence and we explain in the appendix how to extend our construction to every model in the class.

In section~\ref{sect:mode}, we describe the Jukes-Cantor model with CpG
influence, the simplest one of the class of manageable models introduced in~\cite{piau:solv}, 
and its main properties. 
In section~\ref{sect:main}, we summarize our main results on the estimation of the elapsed time between an old DNA sequence and a present one, and on the time since two present DNA sequences issued from the same ancestral sequence diverged. The appendix contains the extension of the results of section~\ref{sect:main}. In the other sections we prove our results.
At the end of section~\ref{sect:main} is a plan of the rest of the paper.

%
\section{Models with influence} \label{sect:mode}

We first describe the Jukes-Cantor model with CpG
influence 
which the results of this paper apply. Then, we mention its main mathematical properties, already established in \cite{piau:solv}, 
and we introduce some notations.

Recall that DNA sequences are encoded by the alphabet
$\AA=\{A,T,C,G\}$, where the letters stand for Adenine, Thymine,
Cytosine and Guanine respectively. 
Thus, bi-infinite DNA sequences are encoded as elements of
$ \AA^{\Z} $ where $\Z$ is the set of integers.

 \subsection{Jukes-Cantor model with CpG
influence (JC+CpG)}

In most models of DNA evolution, one assumes that each site evolves
independently from the others and follows a given Markovian kernel,
see \cite{jukes:69}, \cite{kimu:80}, \cite{fels:81} and \cite{hase:85}
for instance. Even in codon evolution models, see \cite{jones:92}, one
often assumes that different codons evolve independently,  with however some exceptions such as \cite{jensen:prob}. On the other hand, it 
is a well known experimental fact, see \cite{duret:CpG} by example, that the nature of the close
neighbours of a site can modify, notably in some cases, the substitution rates observed at this site. To take account of these observations, 
we consider models, in continuous time, where the sequence
evolves under the combined effect of two superimposed mechanisms. 

The
first mechanism is an independent evolution of the sites as in the
usual models. Hence it is characterized by a $ 4\times4 $ matrix of
substitution rates, each rate being the mean number of substitutions per unit of time.
The simplest case is the Jukes-Cantor model, where each substitution happens at the same rate.
Hence, the rate of the substitutions
of $ x $ by $ y $ is set to $ 1 $, for every nucleotides $x$ and $y$ in $\AA$. 

A second mechanism is superimposed, which describes the substitutions due to the influence of the neighborhood: the
most noticeable case is based on experimentally observed 
CpG-methylation-deamination processes, whose biochemical causes are
well known. Hence we assume that the substitution rates of cytosine by
thymine and of guanine by adenine in CpG dinucleotides are both
increased by an additional nonnegative rate $ r $.

This means for example that any $C$ site whose right neighbour is
not occupied by a $ G $, changes at global rate $ 3 $, hence after an exponentially distributed random 
time with mean $ 1/3 $, and when it does, it becomes an $ A $, a $ G
$ or a $ T $ with probability $ 1/3 $ each. On the contrary, any $C$
site whose right neighbour is occupied by a $ G $, changes at
global rate $ s=3 + r $, hence after an exponentially distributed random time with mean $
1/s $, and when it does, it becomes an $ A $, a $ G $ or a $ T
$ with unequal probabilities $ 1/s $, $ 1/s $, and $
(1+r)/s $ respectively.

The case $ r = 0 $ corresponds to the usual Jukes-Cantor model. As
soon as $r\ne0$, the evolution of a site is not independent of the
rest of the sequence. Hence the evolution of the complete sequence is
Markovian (on a huge state space), but not the evolution of a given
site, nor of any given finite set of sites.

Recall from \cite{piau:solv} that the relevant class of models, called RN+YpR in this paper, is in fact larger than just described.

As already mentioned, the results of this paper about Jukes-Cantor models with CpG influence (hereafter denoted JC+CpG) are adapted to every RN model with YpR influence (hereafter denoted RN+YpR) in the appendix.

\subsection{Main properties}

We work on the space $ \AA^\Z $ with the topology product and the cylindric $\sigma$-algebra defined as the smallest $\sigma$-algebra such that every projection on $ \AA^\Z $ is measurable.

We now recall some results of \cite{piau:solv}, valid for every RN+YpR model. First, for
every probability measure $ \nu $ on $ \AA^\Z $, there exists a unique
Markov process $ (X(t))_{t \geq 0} $ on $ \AA^\Z $, with initial distribution $ \nu $, associated
to the transition rates above. Thus, for every time $t$, $X(t)$ describes the whole sequence and, for every $i$ in $\Z$, the $i$th coordinate 
$X_i(t)$ of $ X(t) $ is the random value of the nucleotide at site $i$ and time $t$.
Under a non-degenaracy condition on the rates of the model, the process $ (X(t))_{t \geq 0} $ is
ergodic, its unique stationary distribution $ \pi $ on $\AA^\Z$ is invariant and
ergodic with respect to the translations of $\Z$, and $ \pi $ puts a
positive mass on every finite word $ w=(w_i)_{0\leq i\leq\ell} $ written in the
alphabet $ \AA $. The notation $\pi(w)$ is abusive because $\pi$ is a measure on $\AA^\Z$ but it is a shorthand for $\pi(\Pi_{0,\ell}^{-1}(\{w\}))$, where $\Pi_{0,\ell}$ is such that for every $ x \in \AA^\Z $, $\Pi_{0,\ell}(x)= (x_i)_{0\leq i\leq\ell}$.

Furthermore, for every
position $i$ in $ \Z $, $ \Prob_\nu(X_{i:i+\ell}(t)=w) $ converges to $ \pi(w) $
when $ t\to+\infty $, where $\Prob_\nu$ stands for the probability under the initial measure $\nu$. Here and later on, for every indices
$ i $ and $ j $ in $ \Z $ with $ i \leq j $ and every symbol $ S $, the
shorthand $ S_{i:j}$ denotes $(S_k)_{i\leq k\leq j} $.  
Finally, if $\xi$ in $\AA^\Z$ is distributed according to $\pi$, the empirical frequencies
of any word $ w $ in $ \xi $, observed along any increasing sequence of
intervals of $ \Z $, almost surely converge to $ \pi(w) $.

All of the above properties stem from the following representation of the distribution $ \pi $. 
There exists an i.i.d.\ sequence $ (\xi_i)_{i\in\Z} $ of Poisson
processes, and a measurable map $ \Psi $ with
values in $ \AA $, such that if one sets
\[
\Xi_i = \Psi(\xi_{i-1},\xi_i,\xi_{i+1})
\]
for every site $i$ in $\Z$, 
then the distribution of $ ( \Xi_i )_{ i\in\Z } $ is $ \pi $.
In particular, any collections $(\Xi_i)_{i\in I}$ and
$(\Xi_i)_{i\in J}$ are independent as soon as the subsets $I$ and $J$ of $ \Z
$ are such that $ |i-j| \geq 3 $ for every sites $ i $ in $ I $ and $ j $
in $ J $. We call this property $2$-dependence.

\subsection{Notations}

Our estimators are based on various quantities provided by the
alignment of the two sequences.

\begin{figure}[ht]
\begin{center}
 \includegraphics{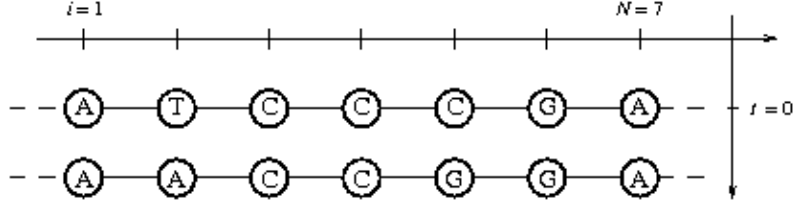}
\end{center}
\caption{Alignment of an ancestral sequence and a present one}
\label{figu:alig}
\end{figure}

For every $\ell\ge0$ and every word $ w $ of length $
\ell + 1 $ written in the alphabet $ \AA $, say that
site $ i $ is occupied at time $ t $ by $ w $ if $
X_{i:i+\ell}(t)=w $.
For every triple of subsets $ W $, $W'$ and $W''$ of words
and every couple of times $ t $ and $ s $,  $ (W)(t) $ denotes the frequency of
sites occupied by any of the words in $ W $ at time $ t $, that is
\[
(W)(t) = \lim_{N\to\infty} \frac1{N} \sum_{i=0}^N \sum_{w\in W} H_i(t,w), \quad \mbox{where} \quad H_i(t,w)= \1\{X_{i:i+\ell}(t)=w\},
\] 
and $( W,W' )(t)$ the
frequency of sites occupied by any of the words in $W$ at time $ 0 $
and any of the words in $ W' $ at time $ t $, that is
\[
( W,W' )(t) = \lim_{N\to\infty} \frac1{N} \sum_{i=0}^N \sum_{w\in W}\sum_{w' \in W'} H_i(0,w) H_i(t,w').
\]

The limits above exist thanks to the ergodicity of $\pi$ with respect to translations.

When comparing two present sequences, we use the following notations. For every sets $ W $ and $ W' $ of words and every time $ t $, $ [W,W'](t) $ denotes the frequency of sites occupied by a word of $ W $ in the left sequence (denoted by $ X^1$) and by a word of $ W' $ in the right sequence (denoted by $ X^2$). 

We identify a word $ w $ and the
set of words $\{w\}$. For every letter $x$ in the alphabet $
\AA $, we use the shorthands $ *x = \AA\times\{x\} $, $x*=\{x\}\times\AA$, $x*x =\AA\times\{x\}\times \AA$  and $ \bar{x} = \AA \setminus \{x\} $.

%
\section{Summary of main results} \label{sect:main}

Our main result is theorem~\ref{theo:finalC} below, which provides asymptotic confidence intervals for an estimation procedure of the time elapsed 
between a present sequence and an ancestral one and for the time since two present sequences issued from the same ancestral sequence diverged, for the Jukes-Cantor model with CpG influence (JC+CpG) of
intensity $r$. These intervals are based on two consistent estimators of the elapsed time and two consistent estimators of the time of divergence. 

Our first estimator is based on the evolution of the frequency $ ( C,C )(t) $ when the time $t$ varies
and the second one on the evolution of $ ( A,A ) (t) $. These estimators match the classic ones used for the original Jukes-Cantor model when $r=0$. The symmetry of the roles played by $A$ and $T$, or by $C$ and $G$ in the JC+CpG model immediately gives the relations $(A,A)(t)=(T,T)(t)$ and $(G,G)(t)=(C,C)(t)$. 

Our estimators for the divergence time are based on the evolution of the frequency $ [ C,C ](t) $ when the time $t$ varies and on the evolution of $ [ A,A ] (t) $. Even if the results are given in the same theorem, there is a substantial difference between $[C,C]$ and $[A,A]$. Indeed, as we explain in sections~\ref{sect:evCA} and~\ref{sect:evac}:
\begin{theo}
In the JC+CpG model, for every positive $t$,
\[
  [C,C](t) = (C,C)(2t), \quad [A,A](t) \neq (A,A)(2t).
\]
\end{theo}

In the appendix, theorem~\ref{theo:finalRN} provides an asymptotic confidence interval for our estimation procedure of the time elapsed 
between a present sequence and an ancestral one, for RN+YpR models, under the condition that the estimator is well-defined in the general case.

The keystep for the creation of phylogenetic trees built by a distance-based method is theorem~\ref{theo:finalC} below. At the moment, a prior knowledge of the parameter $r$ is needed to apply the method. We hope in the future to perform simulations and/or to find a mathematical
method to estimate parameter~$r$.

We now introduce some notations needed to state theorem~\ref{theo:finalC} and used in the rest of the paper.


\begin{defi} \label{defi:xx}
Let $ ( x,x )_\mathrm{obs} $ and $ [x,x]_\mathrm{obs} $ denote for every $ x \in \{A,C\} $ the observed value of 
$ ( x,x ) $ and $ [x,x] $ on two aligned sequences of length $N$, that is,
\[
( x,x )_\mathrm{obs}=\frac1{N} \sum_{i=1}^NK^x_i(t),\quad\mbox{with}\quad
K^x_i(t)=\1\{X_i(0)=X_i(t)=x\} ,
\]
and
\[
[x,x]_{\mathrm{obs}}=\frac1{N} \sum_{i=1}^N  \widetilde{K}^x_i(t),\quad\mbox{with}\quad  \widetilde{K}^x_i(t)=\1\{X_i^1(t)=X_i^2(t)=x\} .
\]
\end{defi}
 
In figure~\ref{figu:alig} for instance, $N=7$ and $ ( C,C )_\mathrm{obs}=\frac27 $.

\begin{defi} \label{defi:TC}
Let $ T_x $ and $\widetilde{T}_x$ denote the estimators of the elapsed time and the divergence time respectively, defined for every $x\in\{A,C\}$,  
as the solution in $t$ of the equations
\[
( x,x ) (t)= ( x,x )_\mathrm{obs} \quad \mbox{and} \quad [ x,x ] (t)= [ x,x ]_\mathrm{obs}.
\]
For $ x \in \{A,C\}$,
let $\kappa^x_\mathrm{obs}$, $ \widetilde{\kappa}^x_{\mathrm{obs}} $, $\nu^x_\mathrm{obs}$ and $ \widetilde{\nu}^x_{\mathrm{obs}} $ denote observed quantities, defined as
\begin{align*}
  \kappa^C_\mathrm{obs} &= 4( C,C )_\mathrm{obs} + r ( C*, CG )_\mathrm{obs} - (C)_{\mathrm{obs}},\\
  \kappa^A_\mathrm{obs} &= 4( A,A )_\mathrm{obs} - r ( *A, CG )_\mathrm{obs} - (A)_{\mathrm{obs}},\\
  \nu^x_\mathrm{obs} &= 
( x,x )_\mathrm{obs} - 5( x,x )_\mathrm{obs}^2 + 2 ( xx,xx )_\mathrm{obs} + 
2 ( x*x, x*x )_\mathrm{obs},
\end{align*}
and
\[
\widetilde{\kappa}^x_{\mathrm{obs}} = 2 \kappa^x_{\mathrm{obs}}, \qquad \widetilde{\nu}^x_{\mathrm{obs}}  = \nu^x_{\mathrm{obs}}.
\]
\end{defi}

We note that $\kappa^x_\mathrm{obs}$, $ \widetilde{\kappa}^x_{\mathrm{obs}} $, $\nu^x_\mathrm{obs}$ and $ \widetilde{\nu}^x_{\mathrm{obs}} $ may be negative for some sequences of observations and some lengths $ N $. However, from lemma~\ref{lemm:lgnCA} in section~\ref{sect:thfi}, $\kappa^x_\mathrm{obs}$, $ \widetilde{\kappa}^x_{\mathrm{obs}} $, $\nu^x_\mathrm{obs}$ and $ \widetilde{\nu}^x_{\mathrm{obs}} $ are almost surely positive when $ N $ is large.

As explained in sections~\ref{sect:evCA} and~\ref{sect:evac}, in the JC+CpG model, for every $ x \in \{A,C\} $, the functions 
$$ 
t \mapsto (x,x)(t),\quad\mbox{and}\quad t \mapsto [x,x](t), 
$$ 
are decreasing functions of $t\ge0$, from $ (x)_* $ at $t=0$ to $ (x)_*^2 $ at $t=+\infty$, where $ (x)_* $ denotes the frequency of $x$ at stationarity. Thus, $ T_x $ and $\widetilde{T}_x$ are unique and well defined for any pair of aligned sequences such that
\[
(x)_*^2<( x,x )_\mathrm{obs}<(x)_*.
\]
Thanks to
the ergodicity of the model, this condition is almost surely satisfied
when $ N $ is large enough because $( x,x )_\mathrm{obs} \to ( x,x )(t)$ and $[ x,x ]_\mathrm{obs} \to [ x,x ](t)$ almost surely
when $ N\to\infty $.

However, even if $T_x$ and $\widetilde{T}_x$ are unique and well defined, the formulas to compute them are not straightforward since they involve inverting a function. Thus, to solve equation $( x,x ) (t)= ( x,x )_\mathrm{obs}$, for example, one has to rely on numerical methods. Fortunately, explicit formulas for $(x,x)(t)$ and $[x,x](t)$ in the JC+CpG model do exist.

We now state our main result.

\begin{theo} \label{theo:finalC}
Assume that the ancestral sequence is at stationarity. Then, in the JC+CpG model, for every $ x \in \{A,C\} $,  
when $ N\to+\infty $, 
$$\kappa^x_\mathrm{obs}\sqrt{N/\nu^x_\mathrm{obs}}(T_x-t)\quad\mbox{and}\quad 
\widetilde{\kappa}^C_\mathrm{obs}\sqrt{N/\widetilde{\nu}^C_\mathrm{obs}}(\widetilde{T}_C - t)
$$ 
both converge in distribution to
  the standard normal law. 
An asymptotic confidence interval at level $ \varepsilon $ for the elapsed time  is
\[
  \left[ 
  T_x -
\frac{z(\varepsilon)}{\kappa^x_\mathrm{obs}}\sqrt{\frac{\nu^x_\mathrm{obs}}{N}}, 
T_x +
\frac{z(\varepsilon)}{\kappa^x_\mathrm{obs}}\sqrt{\frac{\nu^x_\mathrm{obs}}{N}}
\right].
\]
An asymptotic confidence interval at level $ \varepsilon $ for the time of divergence is
\[
  \left[ \widetilde{T}_x - \frac{z(\varepsilon)}{\widetilde{\kappa}^x_\mathrm{obs}}\sqrt{\frac{\widetilde{\nu}^x_\mathrm{obs}}{N}} ,
\widetilde{T}_x + \frac{z(\varepsilon)}{\widetilde{\kappa}^x_\mathrm{obs}}\sqrt{\frac{\widetilde{\nu}^x_\mathrm{obs}}{N}}  \right].
\]
In both formulas, $z(\varepsilon) $ denotes the unique real number such that $\Prob( | Z | \ge z(\varepsilon) ) = \varepsilon$ with $ Z $ a standard normal random variable.
\end{theo}

Theorem~\ref{theo:finalC} implies that, for large $N$, the width of the confidence interval scales as $N^{-1/2}$ times a function of $t$, and that, for large $t$, this function scales as $\ee^{4t}$ for the time elapsed and as $\ee^{8t}$ for the time of divergence, according to formulas given in corollaries~\ref{coro:evoC} and~\ref{coro:evoA}. Heuristically, this means that, to estimate the time $ t $ up to a given factor, one must observe a part of the sequence of length $ N $ at least of order $ \ee^{8t} $ for the time elapsed and at least of order $ \ee^{16t} $ for the time of divergence.

The rest of the paper is organized as follows. In section~\ref{sect:CLT}, we state central limit theorems for the time estimators for the Jukes-Cantor model with CpG influence and for the general model under conjecture~\ref{conj:diff}. 
In section~\ref{sect:thfi}, we show that the central limit theorems established in section~\ref{sect:CLT} imply theorem~\ref{theo:finalC} of section~\ref{sect:main}. 
In section~\ref{sect:evCA}, and~\ref{sect:evac}, we characterize the evolutions of $ ( C,C )(t) $ and $[C,C](t)$, and in section~\ref{sect:evac} the evolutions of $ ( A,A )(t) $and  $ [A,A](t) $. We state some monotonicity properties in these two sections. 

In appendix~\ref{appe:deRN}, we give a short description of the general RN model with YpR influence. In appendix~\ref{appe:resu}, we give an  extension of theorem~\ref{theo:finalC} to the general model under conjecture~\ref{conj:diff}, and in appendix~\ref{appe:just} the justification of this extension. In appendix~\ref{appe:simu}, we describe some simulations supporting our conjecture~\ref{conj:diff}.


\section{Central limit theorems for time estimators} \label{sect:CLT}

We give here central limit theorems for the time estimators in the general model. The strategy is the following. We first deal with $(x,x)_{\mathrm{obs}}$ and $[x,x]_{\mathrm{obs}}$. We compute exactly the variance of these quantities thanks to the $2$-dependence. Then, we use a central limit theorem for mixing sequences. To state central limit therorem for the time estimators, we use the delta method, and to do that, we need to know that $t \mapsto (x,x)(t)$ and $t \mapsto [x,x](t)$ are diffeomorphisms. This is still a conjecture for the general model whereas we prove it for the JC+CpG model.

\subsection{Variance computations} \label{sect:obse}

We detail the properties of $( C,C )_{\mathrm{obs}}$, $( A,A )_{\mathrm{obs}}$, $[C,C]_{\mathrm{obs}}$ and $[A,A]_{\mathrm{obs}}$. We assume that $ N \geq 2 $. 

\begin{lemm} \label{lemm:moyCA}
In the general RN+YpR model, for $ x \in \{C,A\} $, the mean of $ ( x,x )_\mathrm{obs} $, respectively $ [ x,x ]_\mathrm{obs} $, with respect to $\pi$ is $ ( x,x )(t) $,  respectively $[x,x](t)$. 

The
  variances of $ ( x,x )_\mathrm{obs} $ and $ [x,x]_\mathrm{obs}$ with respect to $\pi$ are both equal to $\sigma^2_x(N,t)$, where
\begin{align*}
  N\sigma^2_x(N,t) = &  ( x,x )(t) - ( x,x )(t)^2 
  + 2(1-1/N) \big( ( xx, xx )(t) - ( x,x )(t)^2 \big)+\\
  & \qquad\qquad{} + 2(1-2/N) 
\big( ( x*x, x*x )(t) - ( x,x )(t)^2 \big).
\end{align*}
\end{lemm}

\begin{proof}
  The random variables $ (K^x_i(t))_{i\in\Z} $, respectively  $ (\widetilde{K}^x_i(t))_{i\in\Z} $, are Bernoulli random variables identically
  distributed with respect to $\pi$, 
their common mean is $ ( x,x ) (t) $, respectively $ [x,x](t)$, and $( x,x )_\mathrm{obs}$, respectively $[ x,x ]_\mathrm{obs}$, is the
  empirical mean of the $N$ values $ K^x_i(t) $, respectively $ \widetilde{K}^x_i(t) $, for $ i $ from $1 $ to $ N $. Thus, we
  obtain the value of $ \E(( x,x )_\mathrm{obs}) $, respectively $ \E([ x,x ]_\mathrm{obs}) $,  as $ ( x,x )(t) $, respectively $ [ x,x ](t) $. Furthermore,
\[
N^2 \sigma^2_x(N,t) = 
\sum_{i=1}^N\var(K^x_i(t))+2\sum_{1\leq i<j\leq N}\cov(K^x_i(t),K^x_j(t)).
\]
The variance of each $ K^x_i(t) $ is $
\var(K^x_1(t))=( x,x ) (t)- ( x,x )(t)^2 $. The $ 3 $-dependence, valid for any RN+YpR model,
implies that each covariance for $ |i-j|\geq 3 $ is zero. The
invariance by translation of $\pi$, valid for any RN+YpR model, shows that each of the $ (N-1) $
covariances such that $ i=j- 1 $ is
\[
\cov(K^x_1(t),K^x_2(t))=( xx,xx ) (t)- ( x,x )(t)^2.
\]
Finally, each of the $ (N-2) $ covariances such that $ i=j- 2 $ is
\[
\cov(K^x_1(t),K^x_3(t))=( x*x,x*x ) (t)- ( x,x )(t)^2.
\]
The same arguments hold for the variance of $ [x,x]_\mathrm{obs}$.
This concludes the proof.
\end{proof}

\subsection{Central limit theorems for $(x,x)_\mathrm{obs}$ and $[x,x]_\mathrm{obs}$}

To prove the convergence in distribution to the normal law, we use the
following result.

\begin{theo}[Hall and Heyde \cite{Hall:mart}] \label{theo:hall}
Let $ (V_i)_{i \in \Z} $ denote a stationary, ergodic, centered, square integrable sequence.  Let $ \mathcal{F}_0= \sigma(V_i\,;\, i \leq 0)$ denote the $\sigma$-algebra generated by the random variables $V_i$ for $i\leq0$.
For every positive integer $n$, introduce
\[ 
U_n=\frac1{\sqrt{n}}\sum_{i=1}^n V_i.
\]
Assume that
\begin{itemize}
\item[\rm (i)] for every positive $n$, the series $ \displaystyle{
    \sum_{k\geq1 } \E(V_k \E(V_n | \mathcal{F}_0 ) ) } $ converges,
\item[\rm (ii)] the series $ \displaystyle \sum_{k \geq K} | \E(V_k
  \E(V_n | \mathcal{F}_0 ) ) | $ converges to zero when $ n\to+\infty
  $, uniformly with respect to $ K $.
\end{itemize}
Then $ \E(U_n^2) $ converges to a real number $ \sigma^2\geq0$
when $ n\to+\infty $. Furthermore, if $ \sigma^2 >0$, then
$ U_n/\sqrt{\sigma^2}$
converges in distribution to the standard normal distribution.
\end{theo}

\begin{prop} \label{prop:normCA}
In the general RN+YpR model, for $ x \in \{C,A\}$, when $N\to+\infty$,  
$\sqrt{N} ( ( x,x )_\mathrm{obs}
  - ( x,x )(t) )$ and $\sqrt{N} ( [ x,x ]_\mathrm{obs}
  - [ x,x ](t) )$ both
converge in distribution to the centered normal
  distribution with variance $\sigma_x^2(t)$, where
\[
  \sigma_x^2(t) = ( x,x )(t) + 2 ( xx, xx )(t) + 2 ( x*x, x*x )(t) - 5 ( x,x )(t)^2.
\]
\end{prop}

\begin{proof}
For any RN+YpR model, for $ x \in \{C,A\} $, the sequence $ ( K^x_i(t) )_{i \in \Z} $, respectively $ (\widetilde{K}^x_i)_{i \in \Z}  $, is stationary and ergodic.
  Let $ V_i^x=K^x_i(t)-( x,x )(t) $, respectively $\widetilde{V}_i^x=\widetilde{K}^x_i - [x,x](t)$. This defines a sequence $(V_i^x)_{i\in\Z}$, respectively $ (\widetilde{V}^x_i)_{i \in \Z} $,
  such that the first hypothesis of theorem~\ref{theo:hall} holds. We
  now check conditions (i) et (ii). The $ 3 $-dependence, valid for any RN+YpR model,
  implies that, for every $ n \geq 3 $, $ \E(V_n^x | \mathcal{F}_0^x ) =
  \E(V_n^x) = 0$, respectively $ \E(\widetilde{V}_n^x | \mathcal{\widetilde{F}}_0^x ) =
  \E(V_n^x) = 0$. Hence we only have to check the cases $ n=1 $ and $
  n=2 $.

  For every $ k \geq 3 $, $ V_k^x $, respectively $ \widetilde{V}_k^x $, is independent of $ \mathcal{F}_0^x $, respectively $ \mathcal{\widetilde{F}}_0^x $,
  and $\E(V_n^x | \mathcal{F}_0^x )$, respectively $\E(\widetilde{V}_n^x | \mathcal{\widetilde{F}}_0^x )$, is $\mathcal{F}_0^x$-measurable, respectively $\mathcal{\widetilde{F}}_0^x$-measurable,  hence
\[
  \E(V_k^x \E(V_n^x | \mathcal{F}_0^x ) ) = \E(V_k^x) \E(\E(V_n^x | \mathcal{F}_0^x )) = 0,
\]
and
\[
  \E(\widetilde{V}_k^x \E(\widetilde{V}_n^x | \mathcal{\widetilde{F}}_0^x ) ) = \E(\widetilde{V}_k^x) \E(\E(\widetilde{V}_n^x | \mathcal{\widetilde{F}}_0^x )) = 0.
\]
This implies (i) and (ii), hence theorem~\ref{theo:hall} applies. 

To
compute the asymptotic variance in the theorem, we note that the
variances of $\sqrt{N} (( x,x )_\mathrm{obs} - ( x,x )(t) )$ and $\sqrt{N} ([ x,x ]_\mathrm{obs} - [ x,x ](t) )$ are both
$ N\sigma^2_x(N,t)$, which converges to $ \sigma_x^2(t) $ when $N\to+\infty$.
\end{proof}

\subsection{Central limit theorems for $ T_x$ and $\widetilde{T}_x$}
We describe explicitly the behaviour of $ T_x - t $ and $\widetilde{T}_x - t$.
To state our result, we use the central limit theorems given in proposition~\ref{prop:normCA}, but we now need to treat separately the JC+CpG model
 and the general RN+YpR model.

For $ x \in \{C,A\}$, let $ \mu_x $, respectively $\widetilde{\mu}_x$, denote the inverse function of $ t \mapsto ( x,x )(t) $, respectively $t \mapsto [x,x](t)$. That is, 
$$
t=\mu_x( ( x,x )(t))=\widetilde{\mu}_x([x,x](t)),
$$
and $ \mu_x $ and $\widetilde{\mu}_x$ are both defined on the interval $((x)_*^2,(x)_*]$.


From propositions~\ref{prop:mon1},~\ref{prop:mon2} and~\ref{prop:mon3}, the functions $ t \mapsto ( x,x )(t) $ and $t \mapsto [x,x](t)$ are diffeomorphisms in the JC+CpG model. In the general RN+YpR model, this is only a conjecture, supported by simulations described in appendix~\ref{appe:simu}, showing that indeed, the function $t \mapsto (C,C)(t)$ is decreasing.

\begin{conj} \label{conj:diff}
In the RN+YpR model, for $x \in\{C,A\}$, the functions $ t \mapsto ( x,x )(t) $ and $t \mapsto [x,x](t)$  are diffeomorphisms from $[0,+\infty)$ to $((x)_*^2,(x)_*]$.
\end{conj}
Then,
\[
  T_x = \mu_x(( x,x )_\mathrm{obs}) \quad \mbox{and} \quad t=\mu_x(( x,x )(t)),
\]
and
\[
  \widetilde{T}_x = \widetilde{\mu}_x([ x,x ]_\mathrm{obs}) \quad \mbox{and} \quad t=\widetilde{\mu}_x([ x,x ](t)).
\]
Besides, the derivatives of $ \mu_x $ and $ \widetilde{\mu}_x $, with respect to $t $ are 
\[
  \mu'_x(( x,x )(t))= \frac1{( x,x )'(t)}
  \quad \mbox{and} \quad
  \widetilde{\mu}'_x([ x,x ](t))= \frac1{[ x,x ]'(t)}.
\]
Using the delta method, see \cite{vdvt:asst},
one gets the following result.
\begin{prop}\label{prop:TCA}
In the JC+CpG model, for $ x \in \{C,A\}$, when $N\to+\infty$, $\sqrt{N} (T_x - t )$, respectively $\sqrt{N} (\widetilde{T}_x - t )$, converges in distribution to
  the centered normal distribution with variance
  $ \sigma^2_x(t)/( x,x )'(t)^2  $, respectively $ {\sigma}^2_x(t)/[ x,x ]'(t)^2  $.

Under conjecture~\ref{conj:diff}, the same results hold for the RN+YpR model.
\end{prop}

\section{Proofs of the results of section~\ref{sect:main} for JC + CpG models} \label{sect:thfi}


Proposition~\ref{prop:TCA} yields the variation of $ T_x $ and $\widetilde{T}_x$ around $ t $ for $x\in \{C,A\}$. A priori, to build a confidence interval  for $ t $ from this proposition requires to know the value of $(x,x)'(t)$, respectively $[x,x]'(t)$, and of ${\sigma}^2_x(t) $, which all depend on the
quantity $ t $ to be estimated.

As is customary, Slutsky's lemma (see \cite{vdvt:asst})
allows to bypass this difficulty through the observed
quantities $\kappa^x_\mathrm{obs}$ and $\nu^x_\mathrm{obs}$, respectively $\widetilde{\kappa}^x_\mathrm{obs}$ and $\widetilde{\nu}^x_\mathrm{obs}$, defined in section~\ref{sect:main}. Indeed, Slutsky's lemma states that if two sequences of random variables $(X_N)_N$ and $(Y_N)_N$ are such that $(X_N)_N$ converges in distribution to a random variable $X$ and $(Y_N)_N$ converges in probability to a constant $c$, then the sequence $(X_N Y_N)_N$ converges in distribution to the random variable $cX$.

\begin{lemm} \label{lemm:lgnCA}
In the JC+CpG model, for $ x \in \{C,A\} $, $\kappa^x_\mathrm{obs}\to- ( x,x )'(t)$, $\widetilde\kappa^x_\mathrm{obs}\to- [ x,x ]'(t)$  
and $\nu^x_\mathrm{obs}\to \sigma_x^2(t)$ 
almost surely when $ N\to+\infty $.
\end{lemm}

\begin{proof}
The equalities  
\begin{align*}
  ( C,C )'(t) &= - 4 ( C,C ) (t) - r ( C*, CG ) (t) + (C)_*,\\
  ( A,A )'(t) &= - 4 ( A,A )(t) + r ( *A, CG ) (t) + (A)_*,
\end{align*}
given in sections~\ref{sect:evCA} and~\ref{sect:evac}, and the almost sure convergence of the observed
quantities $( C,C )_\mathrm{obs}$, $( C*, CG )_\mathrm{obs}$,
$( CC,CC )_\mathrm{obs}$, $( C*C,C*C )_\mathrm{obs}$, $( A,A )_\mathrm{obs}$, $( *A, CG )_\mathrm{obs}$, $( AA,AA )_\mathrm{obs}$ 
and $( A*A,A*A )_\mathrm{obs}$ to the
corresponding limiting values, when $N\to+\infty$, 
imply the desired convergences. 
Likewise, the equalities
\begin{align*}
  [ C,C ]'(t) &= - 8 [C,C](t) - 2 r [C*,CG](t) + 2(C)_*,\\
  [ A,A ]'(t) &= - 8 [A,A](t) + 2 r [*A,CG](t) + 2(A)_*,
\end{align*}
imply the convergence of $\widetilde\kappa^x_\mathrm{obs}$.
\end{proof}

We apply Slutsky's lemma to the sequence $(X_N)=(\sqrt{N}(T_x - t))$, respectively $(\widetilde{X}_N)=(\sqrt{N}(\widetilde{T}_x - t))$, which converges in distribution to the centered normal law with variance  $\sigma_x^2(t)/(x,x)'(t)^2$, respectively ${\sigma}_x^2(t)/[x,x]'(t)^2$, from proposition~\ref{prop:TCA}, and the sequence $(Y_N) = (\kappa^x_\mathrm{obs} / \sqrt{\nu^x_\mathrm{obs}})$, respectively  $(\widetilde{Y}_N) = (\widetilde{\kappa}^x_\mathrm{obs} / \sqrt{\widetilde{\nu}^x_\mathrm{obs}})$, which converges in probability to $-( x,x )'(t) / \sigma_x(t)$, respectively  $-[ x,x ]'(t) / {\sigma}_x(t)$, from lemma~\ref{lemm:lgnCA}. This implies theorem~\ref{theo:finalC}.

\section{Evolutions of $( C,C ) (t) $ and $[ C,C ](t)$ in JC+CpG models} \label{sect:evCA}

In the JC+CpG model, dinucleotides coded as $\{ C, \bar C \} \times \{ G, \bar G \} $ have autonomous evolution with the following $4 \times 4$ rate matrix $Q$:
\[
  \bordermatrix{
                    & CG     &\bar{C} G      & \bar{C} \bar{G}    &C \bar{G} \cr
    CG              & -(6+2r)& 3+r           & 0   & 3+r                     \cr
    \bar{C} G       & 1      & -4            & 3   & 0                       \cr
    \bar{C} \bar{G} & 0      & 1             & -2  & 1                       \cr
    C \bar{G}       & 1      & 0             & 3   & -4                      \cr}.
\]
The dynamics of the dinucleotides can be represented with the graph given in figure~\ref{figu:grapC}.

\begin{figure}[!ht] 
\begin{center}
\includegraphics{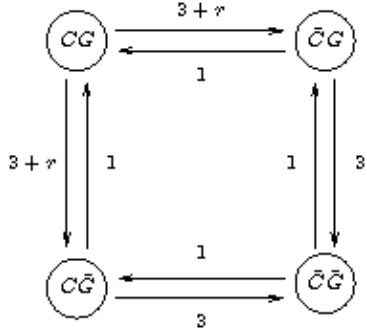}
\end{center}
\caption{Dynamics of dinucleotides encoded as $\{ C, \bar C \} \times \{ G, \bar G \} $}
\label{figu:grapC}
\end{figure}

The exponential of the corresponding matrix can be explicitly computed. One can also compute explicitly the stationary frequencies of dinucleotides coded as $\{ C, \bar C \} \times \{ G, \bar G \} $ using the same matrix. That is
\begin{align*}
 (CG)_* &= \frac{1}{16+5r},  \qquad (C \bar{G})_* = \frac{3+r}{16+5r}, \\
 (\bar{C} \bar{G})_* &= \frac{9+3r}{16+5r}, \qquad (\bar{C} G)_* = \frac{3+r}{16+5r}.
\end{align*}
These stationary frequencies have already been derived in~\cite{piau:solv} by Bérard, Gouéré and Piau.

We observe that $( C,C ) (t)$ can be expressed as a linear combination of terms of the form $( XY,ZT )(t)$ where $(X,Y)$ and $(Z,T)$ belong to $\{ C, \bar C \} \times \{ G, \bar G \} $.

It is then clear that an explicit expression for $( C,C )(t)$ can be obtained, and that an expression of $( C,C ) '(t)$ in terms of dinucleotide frequencies holds.

\begin{prop} \label{theo:evol1}
The evolution of $ ( C,C ) (t) $ satisfies the linear differential equation
\[
  ( C,C )'(t) = -4 ( C,C ) (t) - r ( C*,CG ) (t) + (C)(0).
\]
\end{prop}

Proposition~\ref{theo:evol1} is valid out of equilibrium. We use it at stationarity hence, in particular, for the initial values
\begin{align*}
(C)(0)=(C)_*=&\frac{4+r}{16+5r},\qquad
(CG)(0)=(CG)_*=\frac1{16+5r}.
\end{align*}

The equation in proposition~\ref{theo:evol1} yields expressions of $( C,C ) (t)$. Consider the positive real numbers $u$, $u_+$ and $u_-$ defined as
\[
u=\sqrt{4+2r+r^2},\quad 
u_+=6+r+u,\qquad
u_-=6+r-u.
\]

\begin{coro} \label{coro:evoC}
In the stationary regime,
\[
  ( C,C ) (t) =  c_0 \ee^{-4t} + c_+ \ee^{-u_+ t} + c_- \ee^{-u_-t} + (C)_*^2,
\]
with
\[
  c_0  = 
\frac{3 + r}{2(16+5 r)}
  \quad \mbox{and} \quad c_\pm = 
\frac{3 + r}{4u (16+5r)^2} \left( u(16+3r)\mp (32+14r+3r^2) \right).
\]
\end{coro}
As expected,
\[
c_++c_-+c_0=(C)_*-(C)_*^2.
\]
Furthermore, for every positive $ r $,
\[
4<u_-<5<2r+7<u_+<2r+8.
\]

Although the JC+CpG model is not reversible, the dynamics of dinucleotides encoded as $\{ C, \bar C \} \times \{ G, \bar G \} $ with respect to this model is reversible. This can easily be checked by looking at the cycles in figure~\ref{figu:grapC}.

Reversibility means that the dynamics will look the same whether time runs forward or backward. As a result, given two sequences at stationarity, the probability of data in a state is the same whether one sequence is ancestral to the other or both are descendants of an ancestral sequence at stationarity. Roughly speaking, for every $(X,Y)$ and $(Z,T)$ that belong to $\{ C, \bar C \} \times \{ G, \bar G \} $, going from a $XY$ at time $t$ to $0$ then back to a $ZT$ at time $t$ on another branch, is equivalent to going from a $XY$ to at time $0$ to a $ZT$ at time $2t$.

As a consequence, for every positive $t$, we have
\[
  [C,C](t) = (C,C)(2t).
\]

For every positive $r$, the parameters $c_\pm$ and $c_0$,
are positive. This proves the following proposition.
\begin{prop}\label{prop:mon1}
In the JC+CpG model, the functions $ t\mapsto ( C,C ) (t) $ and $ t\mapsto [ C,C ] (t) $ are
decreasing diffeomorphisms from $[0,+\infty)$ to $((C)_*^2,(C)_*]$.
\end{prop}


\section{Evolutions of $(A,A)(t)$ and $[A,A](t)$ in JC+CpG models} \label{sect:evac}

Like we did to study $(C,C)$, it is possible to encode dinucleotides such that under the JC+CpG model, $(A,A)$ is a linear combination of terms involved in an autonomous evolution. It suffices to encode the dinucleotides as $ \{ C, \bar C \} \times \{A,G,Y\}$, and the dynamics can be represented with the graph given in figure~\ref{figu:grapA}.

\begin{figure}[!ht]
\begin{center}
\includegraphics{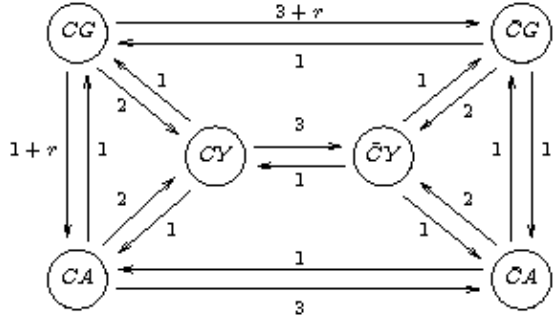}
\end{center}
\caption{Dynamics of dinucleotides encoded as $\{ C, \bar C \} \times \{ A, G, Y \} $}
\label{figu:grapA}
\end{figure}

However, we don't use this encoding to compute $(A,A)(t)$. Indeed, the evolution matrix associated to this encoding is a $6\times 6$ matrix whereas it is possible to deal with the $4\times4$ matrix $Q$, defined in section~\ref{sect:evCA}, to state the evolution of $(A,A)$ as explained below. 

We choose to present this encoding because it is a way to understand the difference between the role of $C$ and $A$ in the Jukes Cantor model with CpG effect. Indeed, the dynamics of dinucleotides encoded as  $ \{ C, \bar C \} \times \{A,G,Y\}$ is not reversible. This can be checked by looking at the cycle $CA \rightarrow CY \rightarrow CG \rightarrow CA $ in figure~\ref{figu:grapA}. As a consequence, even if the non-reversibility of the dynamics does not strictly prove that the identity $ [A,A](t) = (A,A)(2t) $ never holds when $r >0$, the non reversibility of the dynamics can explain why such an identity is unlikely to be true, and in fact, unlike $[C,C]$, as soon as $r > 0$ and $t>0$, 
\[
  [A,A](t) \ne (A,A)(2t).
\]
We strictly explain  this fact at the end of the current section. Now, we describe a way to state the expression of $(A,A)(t)$. Given that there are only three distinct set of two-letter configurations leading to different transition rates to $*A$, that is, $*A$, $CG$, and the complement of these two, the following result is easy to derive.
\begin{prop}
The evolution of $(A,A)(t)$ satisfies the linear differential equation
\[
  (A,A)'(t) = -4(A,A)(t) + r(*A,CG)(t) + (A)(0).
\]
\end{prop}

Let $U(t)$ denote the time dependent vector defined as
\[
  \begin{pmatrix} (*A,CG)(t)\\ (*A , \bar{C}G)(t)\\ (*A , \bar{C}\bar{G})(t) \\ (*A, C\bar{G})(t) \end{pmatrix},
\]
then we have, as a straightforward consequence of the encoding $\{C, \bar C\} \times \{ G, \bar G \} $,
\[
  U'(t) = {}^tQ U(t).
\]
We can now compute $(*A,CG)(t)$, infer the value $(A)_*$ of $(A)(0)$ at stationarity and finally state the expression of $(A,A)(t)$.

\begin{coro} \label{coro:evoA}
In the stationary regime,
\[
  (A,A)(t) =  a_0 \ee^{-4t} + a_+ \ee^{-u_+ t} + a_- \ee^{-u_-t} +  (A)_*^2,
\]
with
\[
  a_0 = \frac{80+31 r}{32(16+5r)},
\]
and
\[ 
  a_\pm = 
\frac{512 + 384 r  + 106 r^2 + 13 r^3\mp u(256+18r+13r^2)}{64u(16+5r)^2}.
\]
\end{coro}

For every positive $r$, the parameters $a_\pm$ and $a_0$,
are positive. This proves the following proposition.
\begin{prop}\label{prop:mon2}
In the stationary JC+CpG model, the function $ t\mapsto ( A,A ) (t) $ is a 
decreasing diffeomorphism from $[0,+\infty)$ to $((A)_*^2,(A)_*]$.
\end{prop}

We deal now with the evolution of $[A,A](t)$. Extending the strategy used to prove proposition~\ref{prop:mon2}, one can also derive an explicit expression (not stated) for $[A,A](t)$, which turns out to be different from $(A,A)(2t)$. Indeed, the computation under \texttt{Maple} shows that the coefficients of $\ee^{-v_+ t}$ and $\ee^{-v_- t}$, where $v_\pm = 10 + r \pm u$, in the expression of $[A,A](t)$ are nonzero. This fact alone proves that $[A,A](t)$ can't be equal to $(A,A)(2t)$. However, we observe on an exemple that the two quantities are very close as one can see on figure~\ref{figu:AA(t)}.

\begin{figure}[!ht] 
\begin{center}
\includegraphics[height=4cm, width=8cm]{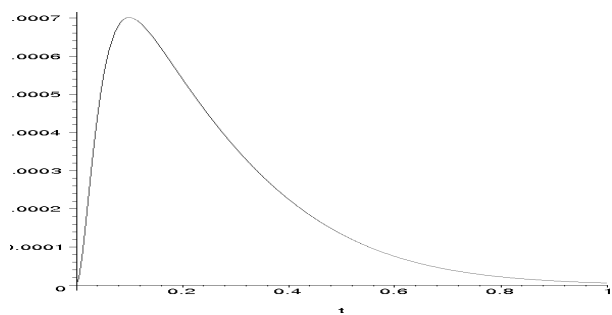}
\end{center}
\caption{Representation of  $t \mapsto [A,A](t) - (A,A)(2t)$, when $r=10$}
\label{figu:AA(t)}
\end{figure}

 We do not provide the expression of $[A,A](t)$, however it seems that the following conjecture holds.
\begin{conj} \label{prop:mon3}
In the JC+CpG model, the function $ t\mapsto [ A,A ] (t) $ is a 
decreasing diffeomorphism from $[0,+\infty)$ to $((A)_*^2,(A)_*]$.
\end{conj}

\section*{Acknowledgments}

I would like to thank an anonymous referee for his deep and thorough reviews, and his helpful comments.

 
\appendix

\section{Short description of the RN model with YpR influence and notations} \label{appe:deRN}

Firstly, RN stands for Rzhetsky-Nei and means that the $ 4\times4 $ matrix of
substitution rates which characterize the independent evolution of the sites must satisfy $4$ equalities, summarized as follows: for every pair of nucleotides $ x $ and $ y \neq x $, the substitution rate from $x$ to  $y$ may depend on $x$ but only through the fact that $x$ is a purine ($A$ or $G$, symbol $R$) or a pyrimidine ($C$ or $T$, symbol $Y$). For instance, the substitution rates from $C$ to $A$ and from $T$ to $A$ must coincide, likewise for the substitution rates from $A$ to $C$ and from $G$ to $C$, from $C$ to $G$ and from $T$ to $G$, and finally from $A$ to $T$ and from $G$ to $T$. The $4$ remaining rates, corresponding to purine-purine  substitutions and to pyrimidine-pyrimidine substitutions, are free.

Secondly, the influence mechanism is called YpR, which stands for the fact that one allows any specific substitution rates between any two YpR dinucleotides ($CG$, $CA$, $TG$ and $TA$) which differ by one position only, for a total of $8$ independent parameters.
The Jukes-Cantor model with CpG effect is the simplest non trivial one: the only YpR substitutions with positive rate are $CG\to CA$ and $CG\to TG$, and both happen at the same rate.

Recall that $Y$ denote the set of pyrimidines defined as $ Y=\{T,C\} $, and $R$ the set of purines defined as $ = \{A,G\}$.

The $ 4\times4 $ matrix of
substitution rates which characterize the independent evolution of the sites in RN model is given by
\[
  \bordermatrix{
      & A     &T      & C    &G      \cr
    A & \cdot & v_T   & v_C   & w_G   \cr
    T & v_A   & \cdot & w_C   & v_G   \cr
    C & v_A   & w_T   & \cdot & v_G   \cr
    G & w_A   & v_T   & v_C   & \cdot \cr}.
\]

The influence mechanism called YpR adds specific rates of substitutions from each YpR dinucleotide as follows.
\begin{itemize}
\item Every dinucleotide $CG$ moves to $CA$ at rate $r_A^C$ and to $TG$ at rate $r_T^G$.
\item Every dinucleotide $TA$ moves to $CA$ at rate $r_C^A$ and to $TG$ at rate $r_G^T$.
\item Every dinucleotide $CA$ moves to $CG$ at rate $r_G^C$ and to $TA$ at rate $r_T^A$.
\item Every dinucleotide $TG$ moves to $CG$ at rate $r_C^G$ and to $TA$ at rate $r_A^T$.
\end{itemize}

\section{Extension of theorem~\ref{theo:finalC} to the RN model with YpR influence} \label{appe:resu}

Under conjecture~\ref{conj:diff}, it is possible to generalize theorem~\ref{theo:finalC} by suitably generalizing the definitions of $\kappa$ and $\nu$ given in section~\ref{sect:main}.
%
Introduce the parameters
\begin{align*}
  \kappa^{RN}_\mathrm{obs} =&\ -v_C (C,A)_\mathrm{obs} - w_C (C,T)_\mathrm{obs} + (v_A+w_T+v_G) (C,C)_\mathrm{obs} - v_C (C,G)_\mathrm{obs}\\
& - r^A_C (C*,TA)_\mathrm{obs} - r^G_C (C*,TG)_\mathrm{obs} + r^A_T (C*,CA)_\mathrm{obs} + r^G_T (C*,CG)_\mathrm{obs}.\\
  \nu^{RN}_{\mathrm{obs}} =&\ \nu^C_{\mathrm{obs}}.
\end{align*}

When $v_C=w_C=v_A=w_T=v_G=1$, $r^A_C=r^G_C=r^A_T=0$ and $r^G_T=r$, which is the case in the JC+CpG model, $\kappa^{RN}_\mathrm{obs}=\kappa^{C}_\mathrm{obs}$.

On the other hand, the observed quantity $\nu^C_\mathrm{obs}$ is unchanged between JC+CpG models and RN+YpR models because lemma~\ref{lemm:moyCA} holds in the general case.

Once again, Slutsky's lemma, through the observed quantities $\kappa^{RN}_\mathrm{obs}$ and $\nu^{RN}_\mathrm{obs}$ is the key to state theorem~\ref{theo:finalRN} below, which is a consequence of proposition~\ref{prop:TCA}.

\begin{theo} \label{theo:finalRN}
Assume that the ancestral sequence is at stationarity and that conjecture~\ref{conj:diff} holds. Then,   
when $ N\to+\infty $, $\kappa^{RN}_\mathrm{obs}\sqrt{N/\nu^{RN}_\mathrm{obs}}(T_C-t)$ converges in distribution to
  the standard normal law. 
An asymptotic confidence interval at level $ \varepsilon $ for $ t $ is
\[
  \left[ 
  T_C -
\frac{z(\varepsilon)}{\kappa^{RN}_\mathrm{obs}}\sqrt{\frac{\nu^{RN}_\mathrm{obs}}{N}}, 
T_C +
\frac{z(\varepsilon)}{\kappa^{RN}_\mathrm{obs}}\sqrt{\frac{\nu^{RN}_\mathrm{obs}}{N}}
\right],
\]
where $z(\varepsilon) $ denotes the unique real number such that $\Prob( | Z | \ge z(\varepsilon) ) = \varepsilon$ with $ Z $ a variable with standard normal law.
\end{theo}

As in the JC+CpG model, the estimator $T_C$ is defined implicitly for RN+YpR models. We do not provide an explicit formula for $(C,C)(t)$ in the general model, but there are numerical methods to compute a closed form of the theoretical solution of the differential linear system, and consequently it is possible to solve equation $(C,C)(t) = (C,C)_{\mathrm{obs}}$ with numerical methods. 


\section{Evolution of $(C,C)(t) $ in RN+YpR models} \label{appe:just}

We base our description of the method in the general RN+YpR model on the encoding of dinucleotides as $\{R,T,C\} \times \{Y,G,A\}$, which has autonomous evolution.

The detailed description of the corresponding $9\times9$ matrix is given below as $m(uv,xy)$, where $uv$ and $xy$ are generic elements of the alphabet.

Let $v_R$ and $v_Y$ denote the quantities defined as
\[
  v_R = v_A + v_G, \quad v_Y = v_T + v_C.
\]
Then,
\begin{align*}
m(uv,xy) &= 0,   
\quad \mbox{if} \quad u \neq x \quad \mbox{and} \quad v \neq y;\\
m(Rx,ux) &= v_u,  
\quad \mbox{if} \quad x \in \{Y,G,A\} \quad \mbox{and} \quad u \in \{C,T\};\\
m(ux,Rx) &= v_R,  
\quad \mbox{if} \quad x \in \{Y,G,A\} \quad \mbox{and} \quad u \in \{C,T\};\\
m(Ru,Rv) &= w_v,
\quad \mbox{if} \quad \{u,v\}=\{A,G\};\\
m(xY,xu) &= v_u,  
\quad \mbox{if} \quad x \in \{R,C,T\} \quad \mbox{and} \quad u \in \{A,G\};\\
m(xu,xY) &= v_R,  
\quad \mbox{if} \quad x \in \{R,C,T\} \quad \mbox{and} \quad u \in \{A,G\};\\
m(uY,vY) &= w_v,
\quad \mbox{if} \quad \{u,v\}=\{T,C\};\\
m(xu,xv) &= w_v + r^x_v, 
\quad \mbox{if} \quad \{u,v\}=\{A,G\} \quad \mbox{and} \quad x \in\{T,C\};\\
m(ux,vx) &= w_v + r^x_v, 
\quad \mbox{if} \quad \{u,v\}=\{C,T\} \quad \mbox{and} \quad x \in\{A,G\}.\\
\end{align*}
It is then clear that quantities such as $(C,C)(t)$ can be computed provided one computes the exponential of the rate-matrix, and that quantities such as $(C,C)'(t)$ have computable explicit expressions in terms of frequencies expressed in the coded dinucleotide-alphabet $\{R,T,C\} \times \{Y,G,A\}$.

\section{Simulations} \label{appe:simu}

As a support to the conjecture that $ t \mapsto (C,C)(t)$ always defines a diffeomorphism in the general RN+YpR model, we performed some simulations. We give the range of parameter values that we explored and one example of figure obtained for one set of parameters, here a Kimura model with CpG influence. The code is available on the website 
\begin{center}
\texttt{http://www-fourier.ujf-grenoble.fr/\~{}mikael.f/en/recherches.htm}
\end{center}
\subsection{Range of parameter values explored}
\[
\begin{array}{||c||c|c|c|c|c|c|c||c||c|c|c|c|c|c|c||}
\hline
 v_A   & 1  & 1  & 1  & 1 & 1 & 1 & 1   & w_A  & 1 & 3  & 0.3  & 0.3 & 3 & 3 & 3\\
\hline
 v_T   & 1  & 1  & 1  & 1 & 1 & 2 & 0.3 & w_T  & 1 & 3  & 0.3  & 0.3 & 3 & 6 & 1\\
\hline
 v_C   & 1  & 1  & 1  & 1 & 1 & 1 & 2   & w_C  & 1 & 3  & 0.3  & 0.3 & 3 & 3 & 1\\
\hline
 v_G   & 1  & 1  & 1  & 1 & 1 & 2 & 10  & w_G  & 1 & 3  & 0.3  & 0.3 & 3 & 6 & 0.1 \\
\hline 
\hline
 r^C_A & 10 & 10 & 10 & 10 & 0.3 & 10 & 10  & r^G_T & 10 & 10 & 10 & 10 & 0.3 & 10 & 5   \\
\hline 
 r^A_C & 0  & 0  & 0  & 10 & 0.3 & 5  & 1   & r^T_G & 0  & 0  & 0  & 10 & 0.3 & 5  & 0.5 \\
\hline
 r^C_G & 0  & 0  & 0  & 10 & 0.3 & 3  & 20  & r^A_T & 0  & 0  & 0  & 10 & 0.3 & 3  & 3   \\
\hline 
 r^G_C & 0  & 0  &  0 & 10 & 0.3 & 1  & 0.3 & r^T_A & 0  & 0  & 0  & 10 & 0.3 & 1  & 0.1 \\
\hline
\end{array}
\]

\subsection{One example of figure performed on Maple}

Figure~\ref{figu:AAkimu} illustrates a simulation performed with the parameter values
\[
  v_A = v_T = v_C = v_G = 1, \quad w_A = w_T = w_C = w_G = 3,
\]
\[ 
  r^C_A = r^G_T = 10, \quad r^A_C = r^T_G = r^C_G = r^A_T = r^G_C = r^T_A = 0.
\]
This is a Kimura model with CpG influence. The function $ t \mapsto [A,A](t) $ is represented on the interval $[0,2]$.

\begin{figure}[!ht] 
\begin{center}
\includegraphics[height=6cm, width=8cm]{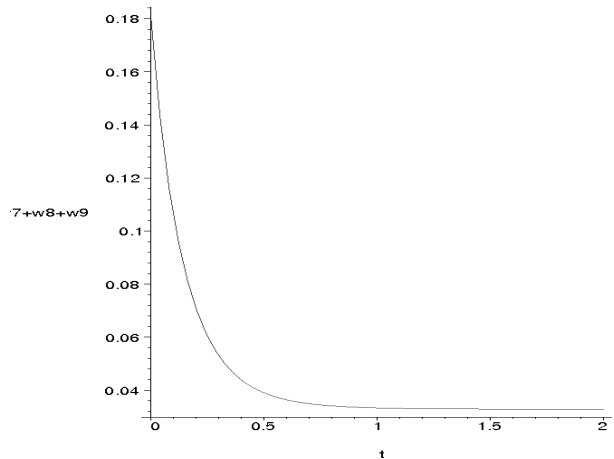}
\end{center}
\caption{One simulation of the function $t \mapsto [A,A](t)$ on the interval $[0,2]$}
\label{figu:AAkimu}
\end{figure}

\bibliographystyle{plain}
\bibliography{references09-05-11}
\end{document}